\documentclass[12pt,reqno]{amsart}
\usepackage[hypertex]{hyperref}
\usepackage[final]{graphicx}
\usepackage{amsfonts}
\usepackage{pdfsync}
\usepackage{amsmath}
\usepackage{amssymb}
\usepackage{amsthm}
\newtheorem{theorem}{Theorem}[section]

\newtheorem{remark}[theorem]{Remark}

\newcommand{\Ha}{{\mathord{\mathbb H}}}
\newcommand{\R}{{\mathord{\mathbb R}}}
\newcommand{\C}{{\mathord{\mathbb C}}}

\begin{document}
\title[trace inequalities]{Sharp trace inequalities for fractional Laplacians}

\author[Einav]{Amit Einav $^1$}

\thanks{A.\ E. \ and M.\ L.\ were supported in part by NSF grant DMS-0901304.}

\author[Loss]{Michael Loss $^2$}

\address{$^1$ Georgia Institute of Technology, School of Mathematics,
Atlanta, Georgia 30332-0160, aeinav@math.gatech.edu}
\email{\href{mailto:aeinav@math.gatech.edu}{aeinav@math.gatech.edu}}

\address{$^2$ Georgia Institute of Technology, School of Mathematics,
Atlanta, Georgia 30332-0160, loss@math.gatech.edu}
\email{\href{mailto:loss@math.gatech.edu}{loss@math.gatech.edu}}

\maketitle

\centerline{\bf Abstract:} The sharp trace inequality of Jos\'e Escobar is extended to traces for the fractional
Laplacian on $\R^n$ and a complete characterization of cases of equality is discussed. The proof proceeds via Fourier transform
and uses Lieb's sharp form of the Hardy-Littlewood-Sobolev inequality.

\section{Introduction}

In a widely cited paper \cite{escobar}, Jos\'e F. Escobar proved the trace inequality 
\begin{equation} \label{grad}
\left( \int_{\R^{n-1}} |(\tau f) (x)|^{\frac{2(n-1)}{n-2}} \right)^{\frac{n-2}{n-1}}  {\rm d} x \le C_n \int_{\Ha^n}  |\nabla f(x,t)|^2 {\rm d}x {\rm d}t \ .
\end{equation}
Here $\Ha^n$ is the $n$--dimensional upper half-space $\{(x,t): x \in \R^n, t >0\}$ and $(\tau f)(x)$ denotes the trace of the function $f$ on the boundary of $\Ha^n$. The constant $C_n$ is given by
\begin{equation}
C_n = \frac{1}{\sqrt \pi (n-2)} \left(\frac{\Gamma(n-1)}{\Gamma(\frac{n-1}{2})}\right)^{\frac{1}{n-1}} \ ,
\end{equation}
and is the sharp constant in (\ref{grad}). Moreover, there is equality in (\ref{grad}) if and only if $f$
is a multiple of the function
\begin{equation}
\frac{ 1}{((a+t)^2 + |x+b|^2)^{\frac{n-2}{2}}}\ , \ a > 0 \ , \ b \in \R^n \ .
\end{equation}
Escobar  observed that the problem is conformally invariant, i.e, the problem can be mapped to the $n$ dimensional
sphere using stereographic projection. He then proceeds by showing first that an optimizer exists and then, using the method of Obata,
to prove that  $f^{\frac{4}{n-2}}s$ ($s$ is the standard metric on Euclidean space) is an Einstein metric and hence flat.
An entirely different proof of this inequality was found by Beckner \cite{beckner}. He deduced this inequality
from the spectral form Lieb's sharp Hardy-Littlewood-Sobolev inequality on the sphere. Inequality (\ref{grad})
is then recovered via stereographic projection.

Another way of getting Lieb's sharp HLS inequality into the picture was shown in \cite{carlenloss}.
It is based on the simple and well known observation that among all functions $f(x,t)$ in the upper half space that have a given trace $g$, the harmonic function
\begin{equation} \label{harmonic}
h(x,t) := e^{-\sqrt{-\Delta}t}g(x) 
\end{equation}
minimizes the Dirichlet integral
\begin{equation}
\int_{\Ha^n} |\nabla f(x,t)|^2 {\rm d} x {\rm d} t \ .
\end{equation}
By Gauss' theorem
\begin{equation}
\int_{\Ha^n} |\nabla h(x,t)|^2 {\rm d} x {\rm d} t = \int_{\partial \Ha^n} \frac{\partial h}{\partial n}(x,0) h(x,0) {\rm d} x
= \int_{\R^{n-1}} g(x) \sqrt{-\Delta} g (x) {\rm d}x =:(g, \sqrt{-\Delta} g) \ .
\end{equation}
Thus, the inequality (\ref{grad}) takes the form of a Sobolev inequality for fractional derivatives
\begin{equation}
C_n (g, \sqrt{-\Delta} g) \ge  \Vert g \Vert^2_{L^{\frac{2(n-1)}{n-2}}(\R^{n-1}} \ ,
\end{equation}
for which Lieb's sharp HLS inequality (\cite{lieb}) yields the sharp constant, including all the cases of equality. 
Another important approach to Escobar's inequality is based on transportation theory \cite{maggivillani}. This approach has been  generalized  in \cite{nazaret} replacing the two norm by general $L^p$--norms.

In view of these developments, it is natural to investigate the sharp trace inequalities for {\it fractional}  Sobolev spaces. More specifically one may ask for the sharp constant for the inequality
\begin{equation}
\left( \int_{\R^{n-1}} |(\tau f) (x)|^{\frac{2(n-1)}{n-2\alpha}} \right)^{\frac{n-2\alpha}{n-1}}  {\rm d} x \le C_{n, \alpha} 
(f, (-\Delta)^{{\alpha}} f)\ ,
\end{equation}
where 
\begin{equation}
(f, (-\Delta)^{\alpha} f) = \int_{\R^n} |\widehat f(k)|^2 (2\pi |k|)^{2\alpha} {\rm d} k \ .
\end{equation}
As usual, 
\begin{equation}
\widehat f(k) = \int_{\R^n} f(x) e^{-2\pi i x \cdot k} {\rm d} x
\end{equation}
denotes the Fourier transform. 
There was an attempt in  \cite{xiao}, where it was shown that for $\alpha \in (0,1)$ and $h$ defined by  (\ref{harmonic})
\begin{equation} \label{xiao}
\Vert g \Vert^2_{L^{\frac{2(n-1)}{n-1-2\alpha}}(\R^{n-1})} \le C_{n,\alpha} \int_{\Ha^n} |\nabla h(x,t)|^2 t^{1-2\alpha} {\rm d}x {\rm d}t
\end{equation}
where 
\begin{equation}
C_{n,\alpha} = \left(\frac{2^{1-4\alpha}}{\pi^{\alpha} \Gamma(2-2\alpha)}\right)\left(\frac{\Gamma((n-1)/2- \alpha)}{\Gamma((n-1)/2+\alpha))}\right)\left(\frac{\Gamma(n-1)}{\Gamma((n-1)/2)}\right)^{\frac{2\alpha}{n-1}} 
\end{equation}
is the sharp constant.  Using the Fourier transform in the variable $x$, it is easy to see that the right side of (\ref{xiao}) can be written as 
$$
2^{2\alpha-1} \Gamma(2-2\alpha) (g, (-\Delta)^{\alpha} g) \ ,
$$
and hence inequality (\ref{xiao}) appears as a Sobolev inequality for the fractional Laplacian on $\R^{n-1}$
but not as one for the fractional Laplacian on $\R^n$.  Note that the harmonic function $h$ does not minimize the right
side of (\ref{xiao}). Minimizing the right side over all functions with a fixed boundary value  $g$ yields a function $u(x,t)$
that satisfies the equation
$$
\Delta_x u + \partial^2_t u +\frac{1-2\alpha}{t} \partial_t u = 0 \ , \ u(x,0)=g(x) \ .
$$
The fractional Laplacian is then recovered by the formula
$$
\lim_{t \to 0} t^{1-2\alpha} u(x,t) = -C(-\Delta)^{\alpha}g \ ,
$$
where $C$ is some constant.
This and general properties of functions that are `harmonic' in this sense have been derived in \cite{caffarellisylvestre}. 

The aim of our little note is to extend the result of Escobar to true fractional Laplacians and prove a class of trace inequalities in their sharp form.  This seems to be justified in view of the popularity of Escobar's result and of the simplicity of our proof of its generalization.

It is customary to define the space $H_\alpha$ to be the space of all functions in $f \in L^2(\R^n)$ whose Fourier  transform $\widehat f$ satisfies
$$
\int_{\R^n} |\widehat f(k)|^2 (1+|k|^2)^\alpha {\rm d}k < \infty \ .
$$
The advantage of this definition is that $f$, being in $L^2(\R^n)$ is automatically a function. Here we are forced to take another route.
For $2 \alpha < n$ we define the space
$D_\alpha(\R^n)$ as the space of tempered distributions in $\mathcal S'(\R^n)$ whose Fourier transform is a function in $L^2(\R^n, |k|^{2\alpha} {\rm d} k)$. More precisely, it consists of elements $T \in \mathcal S'(\R^n)$ with the property that there exists a function
$\widehat f \in L^2(\R^n, |k|^{2\alpha} {\rm d} k)$ such that for all $\phi \in \mathcal S(\R^n)$
$$
\widehat T(\phi) =: T(\check\phi) = \int_{\R^n} \widehat f (k) \phi(k) {\rm d} k \ .
$$

It is easy to see that $D_\alpha$ 
endowed with the norm
\begin{equation}
\Vert f \Vert_{D_\alpha} := \left( \int_{\R^n} |\widehat f(k)|^2 |2 \pi k|^{2\alpha}{ \rm d }k\right)^{1/2}
\end{equation}
is a Hilbert space. Note, the factor $2\pi$ is convenient since we can interpret the right side as
$$
(f, (-\Delta)^\alpha f)
$$
where $(f,g)$ is the standard $L^2(\R^n,  {\rm d}x)$ inner product. It is not difficult to see that $\mathcal S(\R^n)$ is dense in $D_\alpha(\R^n)$ and we will use this fact frequently. We shall see later that the tempered distributions in  $D_\alpha$ are in fact functions.

For $f \in \mathcal S(\R^n)$ we define the restriction of $f$ to the  $n-m$ dimensional hyperplane given by $\{ x \in \R^n: x=(x_1, \dots, x_{n-m}, 0, 0, \dots , 0) \}$
as
$$
 f(x_1, \dots, x_{n-m}, 0, 0, \dots , 0) =: (\tau_m f)(x_1, \dots , x_{n-m}) \ .
 $$
It is a standard result  that $\tau_1$ extends to a bounded operator
from $H_\alpha (\R^n)$ to $H_{\alpha - 1/2}(\R^{n-1})$. For the case where $f \in D_\alpha(\R^n)$ more can be said. 
\begin{theorem}[Sobolev trace inequality]
\label{thm: sharp trace inequality}Let $0\leq  m <n$
and $\frac{m}{2}<\alpha<\frac{n}{2}$. For any $f\in D_\alpha $
we have \begin{equation}
\left\Vert \tau_{m}f \right\Vert _{L^{\frac{2(n-m)}{n-2\alpha}}}^{2}\leq C_{m,\alpha,n} \Vert f \Vert_{D_\alpha}
\label{eq:trace-check inequality for Schwartz functions with general j-1}\end{equation}
where
 \[
C_{m,\alpha ,n}=2^{-2\alpha} \pi^{-\alpha}\cdot\frac{\Gamma\left(n/2-\alpha \right)\Gamma\left(\alpha-m/2 \right)}{\Gamma\left(\alpha\right)\Gamma\left(n/2+\alpha-m \right)}\left\{ \frac{\Gamma\left(n-m\right)}{\Gamma\left((n-m)/2\right)}\right\} ^{(2\alpha-m)/(n-m)} \ .\]
There is equality only if $f(x)$ is proportional to
\begin{equation}\label{minimizer form}
\int_{\R^m} \frac{1}{(|x'|^2 + |x''-y''|^2)^{(n-2\alpha)/2}} \frac{1}{(\gamma ^2 +|y''-a|^2)^{(n+2\alpha-2m)/2}} dy'' \ .
\end{equation}
for some $a \in \R^{n-m}$ and $\gamma \not= 0$.
\end{theorem}
Note that for the case $\alpha =1$ and $m=1$ our constant differs from Esobar's by a factor $1/2$. This is due to the fact that our functions 
are defined on the whole space. For general values of $\alpha$, $D_\alpha(\Ha^n) $ is not defined unless one considers
only functions that have support in $\Ha^n$. Also, in the case of $\alpha=1$ and $m=1$ (\ref{minimizer form}) is easily computable and yields (up to a constant) the optimizer in Escobar's paper \cite{escobar}, as expected.

As mentioned above, trace theorems from  $H_\alpha (\R^n)$ to $H_{\alpha - 1/2}(\R^{n-1})$ are standard. 
Theorem \ref{thm:trace}  in the next section yields the exact norm of the  trace as a linear operator from  $D_\alpha (\R^n)$ to $D_{\alpha - 1/2}(\R^{n-1})$. 
This, together with the sharp Sobolev inequality will yield Theorem  \ref{thm: sharp trace inequality}.
which is, as far as we know, also new for integer $\alpha$,  e.g., $\alpha =2$ and  $n \ge 5$. The tools used in the various proofs, with the exception
of Lieb's sharp Hardy-Littlewood-Sobolev inequality, are all elementary.

\bigskip

{\bf Acknowledgement:} We would like to thank Rupert Frank for helpful discussions concerning this work.

\section{Proof of Theorem \ref{thm: sharp trace inequality}}

That $f \in D_\alpha(\R^n)$ is in fact a function follows among other things from the next  theorem. 
Its proof is patterned after the proof of the Sobolev inequality for relativistic kinetic energy $|p|$ in \cite{analysis}
Theorem 8.4., and is just a dual version of Lieb's sharp Hardy-Littlewood-Sobolev inequality.
\begin{theorem}[Sobolev inequality] \label{thm:sobolev} Let $0\le \alpha < n/2$. For any $f \in D_\alpha(\R^n)$ 
\begin{equation} \label{sobolev}
\Vert f \Vert_s^2 \le    2^{-2\alpha}  \pi^{-\alpha} \frac{\Gamma(n/2-\alpha)}{\Gamma(n/2+\alpha)} \left\{ \frac{\Gamma(n)}{\Gamma(n/2)}\right\}^{2\alpha/n}  \Vert f \Vert_{D_\alpha}^2
\end{equation}
where $s = 2n/(n-2\alpha)$.
The constant on the right side of (\ref{sobolev}) is best possible and there is equality if and only if 
\begin{equation}  \label{optimal}
f(x) = A(\gamma^2 +|x-a|^2)^{-(n-2\alpha)/2}
\end{equation}
where $A \in \C, \gamma \not= 0 $ and $a \in \R^n$.
\end{theorem}
\begin{remark}
As a consequence of the above theorem and the density of $\mathcal S(\R^n)$ in $D_\alpha(\R^n)$ we have that any distribution  in $D_\alpha(\R^n)$ is a function. Thus, we can say that $D_\alpha(\R^n)$ is the space of all functions $f \in
L^{2n/(n-2\alpha)}(\R^n)$ for which there exists a function $g \in L^2(\R^n, |k|^{2\alpha} dk)$ such that
$$
\int_{\R^n}\phi(x) f(x) dx =  \int_{\R^n} \widehat \phi(k) g(k) dk 
$$
for all $\phi \in \mathcal S(\R^n)$. This function $g$ is unique and we shall denote it by $\widehat f(k)$.
\end{remark}
\begin{proof}
For $f, g \in \mathcal S(\R^n)$ we have that
\begin{equation} \label{plancherel}
\int_{\R^n} \overline {f(x)} g(x) {\rm d} x =  \int_{\R^n} \overline {\widehat f (k)} \widehat g(k) {\rm d} k \ ,
\end{equation}
and by Schwarz's inequality
\begin{equation} \label{schwarz}
\int_{\R^n} \overline {\widehat f (k)} \widehat g(k) {\rm d}k \le  \Vert f \Vert_{D_\alpha}\left( \int_{\R^n} \frac{|\widehat g(k)|^2}{|2 \pi k|^{2\alpha}} {\rm d} k\right)^{1/2} \ .
\end{equation}
Since
\begin{equation} \label{equivalence}
\int_{\R^n} \frac{|\widehat g(k)|^2}{|k|^{2\alpha}} {\rm d} k = \pi^{-n/2+2\alpha} \frac{\Gamma(n/2 -\alpha)}{\Gamma(\alpha)} \int_{\R^n} \int_{\R^n} \frac{ \overline {g (x)} g(y)}{|x-y|^{n-2\alpha}} {\rm d} x {\rm d} y \ ,
\end{equation}
(see Corollary 5.10 in \cite{analysis}), we may use Lieb's sharp form of the Hardy-Littlewood-Sobolev inequality  \cite{lieb} (see also \cite{analysis}, \cite{carlenloss}), which states that
\begin{equation}\label{sharpHLS}
 \int_{\R^n} \int_{\R^n} \frac{ \overline {g (x)} g(y)}{|x-y|^{n-2\alpha}} {\rm d} x {\rm d} y \le \pi^{n/2 -\alpha} \frac{\Gamma(\alpha)}{\Gamma(n/2 + \alpha)} \left\{ \frac{\Gamma(n)}{\Gamma(n/2)}\right\}^{2\alpha/n} \Vert g \Vert_r^2
 \end{equation}
where $r = 2n/(n+2\alpha) <2 $. There is equality if and only if
\begin{equation} \label{optimizer}
g(x) =  A (\gamma^2 +|x-a|^2)^{-(n+2\alpha)/2} \
\end{equation}
where $A \in \C, \gamma > 0$ and $a \in \R^n$. Thus, 
$$
\left( \int_{\R^n} \overline{ f(x)} g(x) {\rm d} x \right)^2 \le  2^{-2\alpha}  \pi^{-\alpha} \frac{\Gamma(n/2-\alpha)}{\Gamma(n/2+\alpha)} \left\{ \frac{\Gamma(n)}{\Gamma(n/2)}\right\}^{2\alpha/n}\Vert f \Vert_{D_\alpha} ^2 \Vert g \Vert_r^2
$$
and  taking the supremum over all $g \in \mathcal S(\R^n)$ yields inequality (\ref{sobolev}) for $f \in \mathcal S(\R^n)$. 

Now let $f \in D_\alpha(\R^n)$ be any distribution. Since $\mathcal S(R^n)$ is dense, there exists a sequence of
functions $\widehat f_j \in \mathcal S(\R^n)$ with 
$$
\int_{\R^n} |\widehat f_j(k) - \widehat f(k)|^2 |k|^{2\alpha} dk \to 0
$$
Thus, $f_j$ is a Cauchy sequence in $L^s(\R^n)$ with a limit which we call $\phi_f $. Since $\widehat f_j$
converges to $\widehat f$ in $\mathcal S'(\R^n)$ and since  the Fourier transform
is an isomorphism on $\mathcal S'(\R^n)$, we have that $f_j$ converges to $f$ in $\mathcal S'(\R^n)$.
Hence, $\phi_f =f$ and the inequality (\ref{sobolev}) is valid for all $f \in D_\alpha(\R^n)$. Similar approximation arguments show that
 equation (\ref{plancherel}) continuous to hold
for all functions $f \in D_\alpha(\R^n)$ and all $g \in L^r(\R^n)$. Hence, to establish the cases of equality we need equality in (\ref{schwarz}) which implies that
\begin{equation}\label{Fourier sobolev minimizer}
\widehat f(k) = C \frac{\widehat g(k)}{|k|^{2\alpha}} \ ,
\end{equation}
where $C$ is some constant and where $\widehat g$ is such that
\begin{equation} \label{equalitycondition}
\int_{\R^n} \frac{|\widehat g(k)|^2}{|k|^{2\alpha}}  {\rm d} k < \infty \ .
\end{equation}
Further, there must be equality in the HLS inequality (\ref{sharpHLS}) which mean that $g$ must
be of the form (\ref{optimizer}). This function is integrable and smooth and hence its Fourier transform is bounded
with fast decay. Hence (\ref{equalitycondition}) also holds for this particular function. It remains to show that
$f$ is of the form (\ref{optimal}). To see this we imitate the proof of Corollary 5.9 in \cite{analysis} and obtain that the inverse
Fourier transform of $\widehat g(k)/|k|^{2\alpha}$ is given by
$$
\pi^{- n/2+2\alpha} \frac{\Gamma(n/2-\alpha)}{\Gamma(\alpha)}\int_{\R^n} \frac{g(y)}{|x-y|^{n-2\alpha} }{\rm d} y \ .
$$
Since $g$ is an optimizer for the HLS inequality it must satisfy the Euler-Lagrange equation 
$$
\int_{\R^n} \frac{g(y)}{|x-y|^{n-2\alpha} }{\rm d} y = C g(x)^{r-1}  \ ,
$$
where $C$ is some constant. This yields the form of the optimizer (\ref{optimal}).

\end{proof}

\begin{theorem}[Trace inequality] \label{thm:trace}
Assume that $\frac{m}{2}  < \alpha < \frac{n}{2}$. Then the trace $\tau_m$ has a unique extension to a bounded operator 
$\tau_m : D_\alpha(\R^n) \rightarrow D_{\alpha-m/2}(\R^{n-m})$. Moreover, for any  $f \in D_\alpha(\R^n)$ 
\begin{equation} \label{traceestimate}
\Vert \tau_m f \Vert_{D_{\alpha-m/2}(\R^{n-m})} \le  \frac{1}{2^m \pi^{\frac{m}{2}}} \frac{\Gamma(\frac{2\alpha - m}{2})}{\Gamma(\alpha)} 
\Vert f \Vert_{D_\alpha(\R^n)} \ .
\end{equation}
The constant in this inequality is best possible and there is equality only if
\begin{equation}\label{equalitys}
\widehat f (k_1,k_2) =  C \frac{\widehat g (k_1)}{(|k_1|^2 + |k_2|^2)^{\alpha}}
\end{equation}
with
\begin{equation} \label{equalityl}
\int_{\R^{n-m}} \frac{|\widehat g (k_1)|^2} {|k_1|^{2\alpha-m}} {\rm d} k_1 < \infty \ .
\end{equation}
Here $k_1 \in \R^{n-m}$ and $k_2 \in \R^m$.
\end{theorem}
\begin{proof}
It is sufficient to prove the statement for $m=1$. The rest follows by iterating the result for $m=1$. Let  $f \in \mathcal S(\R^n)$
and note that
$$
 \widehat {(\tau_1 f)}(k')  =  \int_\R  \widehat f(k', k_n) {\rm d} k_n \ .
$$
By Schwarz's inequality,
\begin{eqnarray} \label{setup}
& & |\int_\R  \widehat f(k', k_n) {\rm d} k_n|^2   \nonumber \\
&=& |\int_\R  \widehat f(k', k_n) (|2\pi k'|^2 + (2\pi k_n)^2)^{\alpha/2}   (|2\pi  k'|^2 + (2 \pi k_n)^2)^{-\alpha/2}  {\rm d} k_n|^2 \nonumber \\
&\le& \frac{1}{2 \sqrt \pi} \frac{\Gamma(\alpha -\frac{1}{2})}{\Gamma(\alpha)} |2 \pi k'|^{-2\alpha+1}  \int_\R | \widehat f(k',k_n)|^2(|2 \pi k'|^2 +(2 \pi k_n)^2)^{\alpha} {\rm d} k_n  \ ,
\end{eqnarray}
which establishes (\ref{traceestimate}) for all $f \in \mathcal S(\R^n)$ and hence the continuous extension of $\tau_1$ to all of $D_\alpha(\R^n)$.

Pick any $f \in D_\alpha(\R^n)$ and let $f_j \in S(\R^n)$ converge to $f$ in $D_\alpha(\R^n)$. Hence $\tau_1f_j$ converges to $\tau_1 f$ in $D_{\alpha-1/2}(\R^n)$ and in particular in $L^{2(n-1)/(n-2\alpha)}(\R^{n-1})$
by Theorem \ref{thm:sobolev}. The formula 
$$
 \widehat {(\tau_1 f_j)}(k')  =  \int_\R  \widehat f_j(k', k_n) {\rm d} k_n 
$$
holds for all $j$.  Integrating this against a test function $\phi \in \mathcal S(\R^{n-1})$ yields
$$
\int_{\R^{n-1}} \phi(x') (\tau_1 f_j) (x') d x' = \int_{\R^n} \widehat \phi(k') \widehat f_j(k',k_n) dk' dk_n  \ .
$$
Hence, 
$$
\int_{\R^{n-1}} \phi(x') (\tau_1 f) (x') d x' = \int_{\R^n} \widehat \phi(k') \widehat f(k',k_n) dk' dk_n 
$$
and 
$$
\widehat {(\tau_1 f)}(k') = \int_{\R} \widehat f(k',k_n) d k_n \ .
$$
Thus, the steps in (\ref{setup}) are valid for any $f\in D_{\alpha}\left(\mathbb{R}^{n}\right)$ and the cases of equality follow from the cases of equality in Schwarz's inequality.

\end{proof}

\begin{proof}[Proof of Theorem \ref{thm: sharp trace nequality}] The inequality follows from using Theorem \ref{thm:trace} with $f$ and Theorem \ref{thm:sobolev} with $\tau_{m}f$. In order to have equality we must have that $f$ satisfies (\ref{equalitys}) with (\ref{equalityl}). Also, $\tau_{m}f$ must satisfy (\ref{Fourier sobolev minimizer}) where $g$ is of the form (\ref{optimizer}) with the appropriate dimension. Since 
$$
 \widehat {(\tau_m f)}(k_1)  =  \int_{\R^m}  \widehat{f}(k_1, k_2) {\rm d} k_2 
$$
we conclude that $\widehat{f}$ is of the form (\ref{equalitys}) with $g$ of the form (\ref{optimizer}) with the appropriate dimension in the variables. This leads to  (\ref{minimizer form}).
\end{proof}

\end{document}